\documentclass[10pt]{article}

\usepackage{amsmath,amsbsy,amssymb,amsthm}
\usepackage{a4wide}
\usepackage{graphicx,psfrag}

\newtheorem{Theorem}{Theorem}
\newtheorem{Definition}{Definition}
\newtheorem{Remark}{Remark}
\newtheorem{ex}{Example}

\def\a{{\alpha(t)}}
\def\t{\tau}
\def\LI{{_aI_t^{\a}}}
\def\RI{{_tI_b^{\a}}}
\def\RD{{_tD_b^{\a}}}
\def\LDM{{_a\mathbb{D}_t^{\a}}}
\def\RDM{{_t\mathbb{D}_b^{\a}}}
\def\C{\left(^{-\alpha(t)}_{\quad k}\right)}
\def\D{\left(^{\,\,-\alpha(t)}_{k-n-1}\right)}
\def\DS{\displaystyle}

\begin{document}

\title{An expansion formula with higher-order derivatives\\
for fractional operators of variable order\thanks{This is a preprint 
of a paper whose final and definite form will be published in 
\emph{The Scientific World Journal} ({\tt http://www.hindawi.com/journals/tswj}). 
Submitted 27-Aug-2013; accepted 19-Sept-2013.}}

\author{Ricardo Almeida\\
\texttt{ricardo.almeida@ua.pt}
\and Delfim F. M. Torres\\
\texttt{delfim@ua.pt}}

\date{Center for Research and Development in Mathematics and Applications (CIDMA)\\
Department of Mathematics, University of Aveiro, 3810--193 Aveiro, Portugal}

\maketitle


\begin{abstract}
We obtain approximation formulas for fractional integrals and derivatives
of Riemann--Liouville and Marchaud types with a variable fractional order.
The approximations involve integer-order derivatives only.
An estimation for the error is given. The efficiency of the approximation method
is illustrated with examples. As applications, we show how the obtained results
are useful to solve differential equations and problems of the calculus of variations
that depend on fractional derivatives of Marchaud type.

\bigskip


\smallskip

\end{abstract}


\section{Introduction}

Fractional calculus is a natural extension of the integer-order calculus
by considering derivatives and integrals of arbitrary real
or complex order $\alpha \in \mathbb{K}$, with $\mathbb{K}=\mathbb{R}$
or $\mathbb{K}=\mathbb{C}$. The subject was born from a
famous correspondence between L'Hopital and Leibniz in 1695,
and then developed by many famous mathematicians, like
Euler, Laplace, Abel, Liouville and Riemann, just to mention a few names.
Recently, fractional calculus has call the attention of a vast number of researchers,
not only in mathematics, but also in physics and in engineering, and has proven
to better describe certain complex phenomena in nature \cite{Dalir,Machado}.

Since the order $\alpha$ of the integrals and derivatives may take any value,
another interesting extension is to consider the order not to be a constant
during the process, but a variable $\a$ that depends on time.
This provides an extension of the classical fractional calculus and was introduced
by Samko and Ross in 1993 \cite{SamkoRoss} (see also \cite{Samko}).
The variable order fractional calculus is nowadays recognized as a useful tool,
with successful applications in mechanics, in the modeling of
linear and nonlinear visco-elasticity oscillators,
and other phenomena where the order of the derivative varies with time.
For more on the subject, and applications of it, we mention
\cite{AlmeidaSamko,Coimbra,Coimbra2,Diaz,Lorenzo,Ramirez1,Ramirez2}.
For a numerical approach see, e.g., \cite{Ma,Valerio,Zhuang}.
Results on differential equations and the calculus of variations
with fractional operators of variable order can be found
in \cite{Od1,MyID:268} and references therein.
In this paper we show how fractional derivatives and integrals
of variable order can be approximated by classical integer-order operators.

The outline of the paper is the following. In Section~\ref{sec:FC}
we present the necessary definitions, namely the fractional operators
of Riemann--Liouville and Marchaud of variable order. Some properties
of the operators are also given. The main core of the paper is Section~\ref{sec:theorems},
where we prove the expansion formulas for the considered fractional operators,
with the size of the expansion being the derivative of order $n\in\mathbb{N}$.
In Section~\ref{sec:example} we show the accuracy of our method with some examples
and how the approximations can be applied in different situations
to solve problems involving variable order fractional operators.


\section{Fractional calculus of variable order}
\label{sec:FC}

In the following, the order of the fractional operators
is given by a function $\alpha\in C^1([a,b],]0,1[)$;
$x(\cdot)$ is assumed to ensure convergence
for each of the involved integrals. For a complete and rigorous
study of fractional calculus we refer to \cite{samko0}.

\begin{Definition}
\label{def:1}
Let $x(\cdot)$ be a function with domain $[a,b]$. Then, for $t\in [a,b]$,
\begin{itemize}
\item the left Riemann--Liouville fractional integral of order $\alpha(\cdot)$ is given by
$$
\LI x(t)=\frac{1}{\Gamma(\alpha(t))}\int_a^t (t-\t)^{\alpha(t)-1}x(\t)d\t,
$$

\item the right Riemann--Liouville fractional integral of order $\alpha(\cdot)$ is given by
$$
\RI x(t)=\frac{1}{\Gamma(\alpha(t))}\int_t^b (\t-t)^{\alpha(t)-1}x(\t)d\t,
$$

\item the left Riemann--Liouville fractional derivative of order $\alpha(\cdot)$ is given by
$$
{_aD_t^{\a}} x(t)=\frac{1}{\Gamma(1-\a)}\frac{d}{dt}\int_a^t (t-\t)^{-\alpha(t)}x(\t)d\t,
$$

\item the right Riemann--Liouville fractional derivative of order $\alpha(\cdot)$ is given by
$$
\RD x(t)=\frac{-1}{\Gamma(1-\a)}\frac{d}{dt}\int_t^b (\t-t)^{-\a}x(\t)d\t,
$$

\item the left Marchaud fractional derivative of order $\alpha(\cdot)$ is given by
\begin{equation}
\label{eq:LMD}
\LDM  x(t)=\frac{x(t)}{\Gamma(1-\a)(t-a)^\a}+ \frac{\a}{\Gamma(1-\a)}
\int_a^t \frac{x(t)-x(\t)}{(t-\t)^{1+\alpha(t)}}d\t,
\end{equation}

\item the right Marchaud fractional derivative of order $\alpha(\cdot)$ is given by
$$
\RDM x(t)=\frac{x(t)}{\Gamma(1-\a)(b-t)^\a}+ \frac{\a}{\Gamma(1-\a)}
\int_t^b \frac{x(t)-x(\t)}{(\t-t)^{1+\alpha(t)}} d\t.
$$
\end{itemize}
\end{Definition}

\begin{Remark}
It follows from Definition~\ref{def:1} that
$$
{_aD_t^{\a}} x(t)=\frac{d}{dt} {_aI_t^{1-\a}}x(t)
\quad \mbox{and} \quad {_tD_b^{\a}} x(t)=-\frac{d}{dt} {_tI_b^{1-\a}}x(t).
$$
\end{Remark}

\begin{ex}[See \cite{SamkoRoss}]
\label{ex1}
Let $x$ be the power function $x(t)=(t-a)^\gamma$.
Then, for $\gamma>-1$, we have
\begin{equation*}
\begin{split}
\LI x(t) &= \frac{\Gamma(\gamma+1)}{\Gamma(\gamma+\a+1)}(t-a)^{\gamma+\a},\\
\LDM x(t) &= \frac{\Gamma(\gamma+1)}{\Gamma(\gamma-\a+1)}(t-a)^{\gamma-\a},
\end{split}
\end{equation*}
and
\begin{multline*}
{_aD_t^{\a}} x(t)=\frac{\Gamma(\gamma+1)}{\Gamma(\gamma-\a+1)}(t-a)^{\gamma-\a}\\
-\alpha^{(1)}(t)\frac{\Gamma(\gamma+1)}{\Gamma(\gamma-\a+2)}(t-a)^{\gamma
-\a+1}\left[\ln(t-a)-\psi(\gamma-\a+2)+\psi(1-\a)\right],
\end{multline*}
where $\psi$ is the Psi function, that is, the derivative
of the logarithm of the Gamma function:
$$
\psi(t) =\frac{d}{dt} \ln\left({\Gamma(t)}\right)= \frac{\Gamma'(t)}{\Gamma(t)}.
$$
\end{ex}

From Example~\ref{ex1} we see that ${_aD_t^{\a}} x(t)\not=\LDM  x(t)$.
Also, the symmetry on power functions is violated when we consider
$\LI x(t)$ and ${_aD_t^{\a}} x(t)$, but holds for $\LI x(t)$ and $\LDM  x(t)$.
Later we explain this better, when we deduce the expansion
formula for the Marchaud fractional derivative. In contrast with the
constant fractional order case, the law of exponents fails for fractional integrals
of variable order. However, a weak form holds (see \cite{SamkoRoss}):
if $\beta(t)\equiv \beta$, $t\in ]0,1[$, then
${_aI_t^{\a}}{_aI_t^{\beta}}x(t)={_aI_t^{\a+\beta}}x(t)$.


\section{Expansion formulas with higher-order derivatives}
\label{sec:theorems}

The main results of the paper provide approximations of the fractional
derivatives of a given function $x$ by sums involving only integer derivatives of $x$.
The approximations use the generalization of the binomial coefficient formula to real numbers:
$$
\C (-1)^k=\left(^{\alpha(t)+k-1}_{\quad\quad k}\right)=\frac{\Gamma(\a+k)}{\Gamma(\a) k!}.
$$

\begin{Theorem}
\label{Mainthm}
Fix $n \in \mathbb{N}$ and $N\geq n+1$, and let $x(\cdot)\in C^{n+1}([a,b],\mathbb{R})$.
Define the (left) moment of $x$ of order $k$ by
$$
V_k(t) =(k-n)\int_a^t (\t-a)^{k-n-1}x(\t)d\t.
$$
Then,
$$
{_aD_t^{\a}} x(t)=S_1(t)-S_2(t)+E_{1,N}(t)+E_{2,N}(t)
$$
with
\begin{equation}
\label{eq:S1}
S_1(t)=(t-a)^{-\a}\left[\sum_{k=0}^{n}A(\a,k)(t-a)^{k}x^{(k)}(t)
+\sum_{k=n+1}^{N}B(\a,k)(t-a)^{n-k}V_k(t)\right],
\end{equation}
where
\begin{equation*}
\begin{split}
A(\a,k)&=\frac{1}{\Gamma(k+1-\a)}\left[1
+\sum_{p=n+1-k}^{N}\frac{\Gamma(p-n+\a)}{\Gamma(\a-k)(p-n+k)!}\right],\quad k = 0,\ldots,n,\\
B(\a,k)&=\frac{\Gamma(k-n+\a)}{\Gamma(-\a)\Gamma(1+\a)(k-n)!},
\end{split}
\end{equation*}
and
\begin{multline}
\label{eq:S2}
S_2(t)=\frac{x(t)\alpha^{(1)}(t)}{\Gamma(1-\a)}(t-a)^{1-\a}\left[
\frac{\ln(t-a)}{1-\a}-\frac{1}{(1-\a)^2}\right.\\
\left.-\ln(t-a)\sum_{k=0}^N \C
\frac{(-1)^k}{k+1}+\sum_{k=0}^N\C(-1)^k
\sum_{p=1}^N\frac{1}{p(k+p+1)}\right]\\
+\frac{\alpha^{(1)}(t)}{\Gamma(1-\a)}(t-a)^{1-\a}\left[
\ln(t-a)\sum_{k=n+1}^{N+n+1}\D\frac{(-1)^{k-n-1}}{k-n}(t-a)^{n-k}V_{k}(t)\right.\\
\left.-\sum_{k=n+1}^{N+n+1}\D(-1)^{k-n-1}
\sum_{p=1}^N\frac{1}{p(k+p-n)}(t-a)^{n-k-p} V_{k+p}(t)\right].
\end{multline}
The error of the approximation ${_aD_t^{\a}} x(t) \approx S_1(t)-S_2(t)$
is given by $E_{1,N}(t)+E_{2,N}(t)$, where $E_{1,N}(t)$
and $E_{2,N}(t)$ are bounded by
\begin{equation}
\label{eq:err1}
|E_{1,N}(t)|\leq L_{n+1}(t)\frac{\exp((n-\a)^2+n-\a)}{\Gamma(n+1-\a)(n-\a)N^{n-\a}}(t-a)^{n+1-\a}
\end{equation}
and
\begin{equation}
\label{eq:err2}
|E_{2,N}(t)|\leq  \frac{L_1(t) \left|\alpha^{(1)}(t)\right|(t-a)^{2-\alpha(t)}
\exp(\alpha^2(t)-\a)}{\Gamma(2-{\a}){N^{1-\a}}}  \left[\left|\ln(t-a)\right|+\frac{1}{N}\right]
\end{equation}
with
$$
L_{j}(t)=\displaystyle\max_{\t \in [a,t]}\left|x^{(j)}(\t)\right|, \quad j\in\{1,n+1\}.
$$
\end{Theorem}

\begin{proof}
Starting with equality
$$
{_aD_t^{\a}} x(t)=\frac{1}{\Gamma(1-\a)}
\frac{d}{dt}\int_a^t (t-\t)^{-\alpha(t)}x(\t)d\t,
$$
doing the change of variable $t-\t=u-a$ over the integral,
and then differentiating it, we get
\begin{equation*}
\begin{split}
{_aD_t^{\a}} x(t)
&=\displaystyle\frac{1}{\Gamma(1-\a)}\frac{d}{dt}\int_a^t (u-a)^{-\alpha(t)}x(t-u+a)du\\
&=\displaystyle\frac{1}{\Gamma(1-\a)}\left[\frac{x(a)}{(t-a)^{\a}}
+\int_a^t \frac{d}{dt}\left[ (u-a)^{-\alpha(t)}x(t-u+a)\right]du\right]\\
&=\displaystyle\frac{1}{\Gamma(1-\a)}\left[\frac{x(a)}{(t-a)^{\a}}\right.\\
&\ +\DS\int_a^t \left[ -\alpha^{(1)}(t)(u-a)^{-\alpha(t)}
\ln(u-a)x(t-u+a)\left.\DS+(u-a)^{-\alpha(t)}x^{(1)}(t-u+a)\right]du\right]\\
&=S_1(t)-S_2(t)
\end{split}
\end{equation*}
with
\begin{equation}
\label{eq:S1:proof}
S_1(t) = \DS\frac{1}{\Gamma(1-\a)}\left[\frac{x(a)}{(t-a)^{\a}}+\int_a^t (t-\t)^{-\alpha(t)}x^{(1)}(\t)d\t\right]
\end{equation}
and
\begin{equation}
\label{eq:S2:proof}
S_2(t)=\DS\frac{\alpha^{(1)}(t)}{\Gamma(1-{\a})}\int_a^t (t-\t)^{-\alpha(t)}\ln(t-\t)x(\t)d\t.
\end{equation}
The equivalence between \eqref{eq:S1:proof} and \eqref{eq:S1} follows
from the computations of \cite{Pooseh1}. To show the equivalence
between \eqref{eq:S2:proof} and \eqref{eq:S2}
we start in the same way as done in \cite{Atanackovic1}, to get
\begin{equation*}
\begin{split}
S_2(t)&=\DS\frac{\alpha^{(1)}(t)}{\Gamma(1-{\a})}\left[x(t)\int_a^t (t-u)^{-\alpha(t)}\ln(t-u)du\right.\\
&\DS-\left.\int_a^t x^{(1)}(\t)\left(\int_a^\t (t-u)^{-\alpha(t)}\ln(t-u)du\right)d\t\right]\\
&=\DS\frac{\alpha^{(1)}(t)}{\Gamma(1-{\a})}\left[x(t)(t-a)^{1-\alpha(t)}\left[
\frac{\ln(t-a)}{1-\a}-\frac{1}{(1-\a)^2}\right]\right.\\
&\DS-\left.\int_a^t x^{(1)}(\t)\left(\int_a^\t (t-a)^{-\alpha(t)}\left(1-\frac{u-a}{t-a}\right)^{-\alpha(t)}\left[
\ln(t-a)+\ln \left(1-\frac{u-a}{t-a}\right)  \right]du\right)d\t\right].
\end{split}
\end{equation*}
Now, applying Taylor's expansion over
$\DS \left(1-\frac{u-a}{t-a}\right)^{-\alpha(t)}$
and $\DS \ln \left(1-\frac{u-a}{t-a}\right)$, we deduce that
\begin{equation*}
\begin{split}
S_2(t)&=\DS\frac{\alpha^{(1)}(t)}{\Gamma(1-{\a})}\left[
x(t)(t-a)^{1-\alpha(t)}\left[\frac{\ln(t-a)}{1-\a}-\frac{1}{(1-\a)^2}\right]\right.\\
&\quad \DS-\int_a^t x^{(1)}(\t)\left(\int_a^\t (t-a)^{-\alpha(t)}\ln(t-a)
\sum_{k=0}^N \C (-1)^k \frac{(u-a)^k}{(t-a)^k}du\right.\\
&\quad \DS\left.\left. -\int_a^\t (t-a)^{-\alpha(t)}\sum_{k=0}^N \C (-1)^k
\frac{(u-a)^k}{(t-a)^k} \, \sum_{p=1}^N \frac{1}{p}
\frac{(u-a)^p}{(t-a)^p} du \right)d\t\right]+E_{2,N}(t)
\end{split}
\end{equation*}
\begin{equation*}
\begin{split}
&=\DS\frac{\alpha^{(1)}(t)}{\Gamma(1-{\a})}\left[x(t)(t-a)^{1-\alpha(t)}\left[
\frac{\ln(t-a)}{1-\a}-\frac{1}{(1-\a)^2}\right]\right.\\
&\quad \DS-\int_a^t x^{(1)}(\t)(t-a)^{-\alpha(t)}\ln(t-a)
\sum_{k=0}^N \C \frac{(-1)^k}{(t-a)^k} \left(\int_a^\t (u-a)^k \, du\right)d\t\\
&\quad \DS\left. +\int_a^t x^{(1)}(\t) (t-a)^{-\alpha(t)}\sum_{k=0}^N \C
\frac{(-1)^k}{(t-a)^k} \, \sum_{p=1}^N \frac{1}{p(t-a)^p}
\left(\int_a^\t (u-a)^{k+p} \, du\right)\,d\t\right]+E_{2,N}(t)\\
&=\DS\frac{\alpha^{(1)}(t)(t-a)^{-\alpha(t)}}{\Gamma(1-{\a})}\left[
x(t)(t-a)\left[\frac{\ln(t-a)}{1-\a}-\frac{1}{(1-\a)^2}\right]\right.\\
&\quad \DS-\ln(t-a) \sum_{k=0}^N \C \frac{(-1)^k}{(t-a)^k(k+1)}
\left(\int_a^t x^{(1)}(\t) (\t-a)^{k+1} \, d\t\right)\\
&\quad \DS\left. +\sum_{k=0}^N \C \frac{(-1)^k}{(t-a)^k} \,
\sum_{p=1}^N \frac{1}{p(t-a)^p(k+p+1)} \left(\int_a^t
x^{(1)}(\t) (\t-a)^{k+p+1} \, d\t \right)\right]+E_{2,N}(t).
\end{split}
\end{equation*}
Integrating by parts, we conclude with the two following equalities:
\begin{equation*}
\begin{split}
\DS\int_a^t x^{(1)}(\t) (\t-a)^{k+1} \, d\t
&=\DS x(t)(t-a)^{k+1}-V_{k+n+1}(t),\\
\DS\int_a^t x^{(1)}(\t) (\t-a)^{k+p+1} \, d\t
&=\DS x(t)(t-a)^{k+p+1}-V_{k+p+n+1}(t).
\end{split}
\end{equation*}
The deduction of relation \eqref{eq:S2} for $S_2(t)$ follows now from direct calculations.
To end, we prove the upper bound formula for the error.
The bound \eqref{eq:err1} for the error $E_{1,N}(t)$ at time $t$
follows easily from \cite{Pooseh1}. With respect to sum $S_2$, the error at $t$ is bounded by
\begin{multline*}
|E_{2,N}(t)|\leq \left| \frac{\alpha^{(1)}(t)(t-a)^{-\alpha(t)}}{\Gamma(1-{\a})}\right|
\times\left|-\ln(t-a) \sum_{k=N+1}^\infty \C \frac{(-1)^k}{k+1} \left(\int_a^t x^{(1)}(\t)
\frac{(\t-a)^{k+1}}{(t-a)^k} \, d\t\right)\right.\\
\left.+\sum_{k=N+1}^\infty \C (-1)^k \sum_{p=N+1}^\infty \frac{1}{p(k+p+1)}
\left(\int_a^t x^{(1)}(\t)  \frac{(\t-a)^{k+p+1}}{(t-a)^{k+p}}\, d\t \right)\right|.
\end{multline*}
Define the quantities
\begin{equation*}
I_1(t) = \DS \int_N^\infty\frac{1}{k^{1-\a}(k+1)(k+2)}\,dk
\end{equation*}
and
\begin{equation*}
I_2(t) = \DS\int_N^\infty\int_N^\infty\frac{1}{k^{1-\a}p(k+p+1)(k+p+2)}\,dp\, dk.
\end{equation*}
Inequality \eqref{eq:err2} follows from relation
$$
\left|\C\right|\leq\frac{\exp(\alpha^2(t)-\a)}{k^{1-\a}}
$$
and the upper bounds
$$
I_1(t)<\int_N^\infty\frac{1}{k^{2-\a}}\,dk=\frac{1}{(1-\a)N^{1-\a}}
$$
and
$$
I_2(t)<\int_N^\infty\int_N^\infty\frac{1}{k^{2-\a}p^2}\,dp\, dk=\frac{1}{(1-\a)N^{2-\a}}
$$
for $I_1$ and $I_2$.
\end{proof}

Similarly as done in Theorem~\ref{Mainthm}
for the left Riemann--Liouville fractional derivative,
an approximation formula can be deduced for the
right Riemann--Liouville fractional derivative:

\begin{Theorem}
\label{MainthmR}
Fix $n \in \mathbb{N}$ and $N\geq n+1$, and let $x(\cdot)\in C^{n+1}([a,b],\mathbb{R})$.
Define the (right) moment of $x$ of order $k$ by
$$
W_k(t) =(k-n)\int_t^b (b-\t)^{k-n-1}x(\t)d\t.
$$
Then,
$$
\RD x(t)=S_1(t)+S_2(t)+E_{1,N}(t)+E_{2,N}(t)
$$
with
$$
S_1(t)=(b-t)^{-\a}\left[\sum_{k=0}^{n}A(\a,k)(b-t)^{k}x^{(k)}(t)
+\sum_{k=n+1}^NB(\a,k)(b-t)^{n-k}W_k(t)\right],
$$
where
\begin{equation*}
\begin{split}
A(\a,k)&=\frac{(-1)^k}{\Gamma(k+1-\a)}\left[1
+\sum_{p=n+1-k}^N\frac{\Gamma(p-n+\a)}{\Gamma(\a-k)(p-n+k)!}\right],
\quad k = 0,\ldots,n,\\
B(\a,k)&=\frac{(-1)^{n+1}\Gamma(k-n+\a)}{\Gamma(-\a)\Gamma(1+\a)(k-n)!},
\end{split}
\end{equation*}
and
\begin{multline*}
S_2(t)=\frac{x(t)\alpha^{(1)}(t)}{\Gamma(1-\a)}(b-t)^{1-\a}
\left[\frac{\ln(b-t)}{1-\a}-\frac{1}{(1-\a)^2}\right.\\
\left.-\ln(b-t)\sum_{k=0}^N \C
\frac{(-1)^k}{k+1}+\sum_{k=0}^N\C(-1)^k\sum_{p=1}^N\frac{1}{p(k+p+1)}\right]\\
+\frac{\alpha^{(1)}(t)}{\Gamma(1-\a)}(b-t)^{1-\a}\left[\ln(b-t)
\sum_{k=n+1}^{N+n+1}\D\frac{(-1)^{k-n-1}}{k-n}(b-t)^{n-k}W_{k}(t)\right.\\
\left.-\sum_{k=n+1}^{N+n+1}\D(-1)^{k-n-1}\sum_{p=1}^N\frac{1}{p(k+p-n)}(b-t)^{n-k-p} W_{k+p}(t)\right].
\end{multline*}
The error of the approximation $\RD x(t) \approx S_1(t)+S_2(t)$
is given by $E_{1,N}(t)+E_{2,N}(t)$, where $E_{1,N}(t)$
and $E_{2,N}(t)$ are bounded by
$$
|E_{1,N}(t)|\leq L_{n+1}(t)\frac{\exp((n-\a)^2+n-\a)}{\Gamma(n+1-\a)(n-\a)N^{n-\a}}(b-t)^{n+1-\a}
$$
and
$$
|E_{2,N}(t)|\leq  \frac{L_1(t) \left|\alpha^{(1)}(t)\right|(b-t)^{2-\alpha(t)}
\exp(\alpha^2(t)-\a)}{\Gamma(2-{\a}){N^{1-\a}}}
\left[\left|\ln(b-t)\right|+\frac{1}{N}\right]
$$
with
$$
L_{j}(t)=\displaystyle\max_{\t \in [a,t]}\left|x^{(j)}(\t)\right|,
\quad j\in\{1,n+1\}.
$$
\end{Theorem}

Using the techniques presented in \cite{Pooseh0},
similar formulas as the ones given by Theorem~\ref{Mainthm}
and Theorem~\ref{MainthmR} can be proved for the left
and right Riemann--Liouville fractional integrals
of order $\alpha(\cdot)$. For example,
for the left fractional integral one has the following result.

\begin{Theorem}
\label{thm:RL:LFI}
Fix $n \in \mathbb{N}$ and $N\geq n+1$,
and let $x(\cdot)\in C^{n+1}([a,b],\mathbb{R})$. Then,
$$
\LI x(t)=(t-a)^{\a}\left[\sum_{k=0}^{n}A(\a,k)(t-a)^{k} x^{(k)}(t)
+\sum_{k=n+1}^N B(\a,k)(t-a)^{n-k}V_k(t)\right]+E_{N}(t),
$$
where
\begin{equation*}
\begin{split}
A(\a,k)&= \frac{1}{\Gamma(k+1+\a)}\left[1
+\sum_{p=n+1-k}^N\frac{\Gamma(p-n-\a)}{\Gamma(-\a-k)(p-n+k)!}\right],
\quad k = 0, \ldots, n,\\
B(\a,k) &= \frac{\Gamma(k-n-\a)}{\Gamma(\a)\Gamma(1-\a)(k-n)!},\\
V_k(t)&=\DS (k-n) \int_a^t (\tau-a)^{k-n-1}x(\tau)d\tau, \quad k=n+1,\ldots
\end{split}
\end{equation*}
A bound for the error $E_{N}(t)$ is given by
$$
|E_{N}(t)|\leq L_{n+1}(t)\frac{\exp((n+\a)^2+n+\a)}{\Gamma(n+1+\a)(n+\a)
N^{n+\a}}(t-a)^{n+1+\a}.
$$
\end{Theorem}

We now focus our attention to the left Marchaud fractional derivative $\LDM  x(t)$.
Splitting the integral \eqref{eq:LMD}, we deduce that
$$
\LDM  x(t)=- \frac{\a}{\Gamma(1-\a)}
\int_a^t \frac{x(\t)}{(t-\t)^{1+\alpha(t)}}d\t.
$$
Integrating by parts,
\begin{equation}
\label{marchaurepre}
\LDM x(t)=\frac{1}{\Gamma(1-\a)}\left[
\frac{x(a)}{(t-a)^\a}+\int_a^t (t-\tau)^{-\a}x^{(1)}(\tau)d\tau \right],
\end{equation}
which is a representation for the left Riemann--Liouville fractional derivative
when the order is constant, that is, $\a \equiv \alpha$ \cite[Lemma~2.12]{samko0}.
For this reason, the Marchaud fractional derivative is more suitable as the inverse operation
for the Riemann--Liouville fractional integral. With Eq.~\eqref{marchaurepre}
and Theorem~\ref{Mainthm} in mind, it is not difficult
to obtain the corresponding formula for $\LDM  x(t)$.

\begin{Theorem}
\label{MainthmM}
Fix $n \in \mathbb{N}$ and $N\geq n+1$, and let
$x(\cdot)\in C^{n+1}([a,b],\mathbb{R})$. Then,
$$
\LDM x(t)=S_1(t)+E_{1,N}(t),
$$
where $S_1(t)$ and $E_{1,N}(t)$ are as in Theorem~\ref{Mainthm}.
\end{Theorem}

Similarly, having into consideration that
$$
\RDM x(t)=\frac{1}{\Gamma(1-\a)}\left[ \frac{x(b)}{(b-t)^\a}
-\int_t^b (\t-t)^{-\a}x^{(1)}(\tau)d\tau \right],
$$
the following result holds.

\begin{Theorem}
Fix $n \in \mathbb{N}$ and $N\geq n+1$, and let
$x(\cdot)\in C^{n+1}([a,b],\mathbb{R})$. Then,
$$
\RDM x(t)=S_1(t)+E_{1,N}(t),
$$
where $S_1(t)$ and $E_{1,N}(t)$ are as in Theorem~\ref{MainthmR}.
\end{Theorem}


\section{Examples}
\label{sec:example}

For illustrative purposes, we consider
the left Riemann--Liouville fractional integral
and the left Riemann--Liouville and Marchaud fractional derivatives
of order $\a=(t+1)/4$. Similar results as the ones presented here
are easily obtained for the other fractional operators and for other
functions $\alpha(\cdot)$. All computations were done using the
Computer Algebra System \textsf{Maple}.


\subsection{Test function}

We test the accuracy of our approximations with an example.

\begin{ex}
Let $x$ be the function $x(t)=t^4$ with $t\in[0,1]$. Then, for $\a=(t+1)/4$,
it follows from Example~\ref{ex1} that
\begin{equation}
\label{ex1:LFI}
{_0I_t^{\a}} x(t)=\frac{24}{\Gamma(\frac{t+21}{4})}t^{\frac{t+17}{4}},
\end{equation}
\begin{equation}
\label{ex1:LFD}
{_0D_t^{\a}} x(t)=\frac{24}{\Gamma(\frac{19-t}{4})}t^{\frac{15-t}{4}}
-\frac{6}{\Gamma(\frac{23-t}{4})}t^{\frac{19-t}{4}} \left[\ln(t)
-\psi\left(\frac{23-t}{4}\right)+\psi\left(\frac{3-t}{4}\right)\right],
\end{equation}
and
\begin{equation}
\label{ex1:LFD:M}
{_0\mathbb{D}_t^{\a}} x(t)
=\frac{24}{\Gamma(\frac{19-t}{4})}t^{\frac{15-t}{4}}.
\end{equation}
In Figures~\ref{fig:1}, \ref{fig:2} and \ref{fig:3} one can compare the exact expressions
of the fractional operators of variable order \eqref{ex1:LFI}, \eqref{ex1:LFD} and \eqref{ex1:LFD:M},
respectively, with the approximations obtained from our results of Section~\ref{sec:theorems}
with $n=2$ and $N \in \{3,5\}$. The error $E$ is measured using the norm
\begin{equation}
\label{eq:err}
E(f,g)(t)=\sqrt{\int_0^1 (f(t)-g(t))^2\,dt}.
\end{equation}
\begin{figure}[!ht]
\begin{center}
\psfrag{x\(t\)=t^4}{\hspace*{-0.5cm}\tiny ${_0I_t^{\a}} x(t)$}
\includegraphics[scale=0.4,angle=-90]{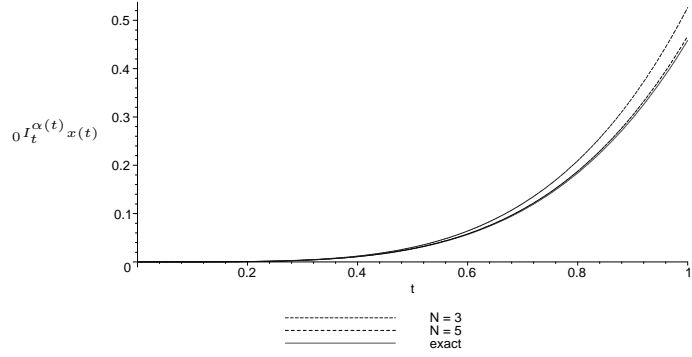}\\
\caption{Exact \eqref{ex1:LFI} and numerical approximations of the left Riemann--Liouville integral
${_0I_t^{\a}} x(t)$ with $x(t)=t^4$ and $\a=(t+1)/4$ obtained from Theorem~\ref{thm:RL:LFI}
with $n=2$ and $N \in \{3,5\}$. The error \eqref{eq:err} is $E\approx 0.02169$
for $N=3$ and $E\approx 0.00292$ for $N=5$.}
\label{fig:1}
\end{center}
\end{figure}
\begin{figure}[!ht]
\begin{center}
\psfrag{x\(t\)=t^4}{\hspace*{-0.5cm}\tiny ${_0D_t^{\a}} x(t)$}
\includegraphics[scale=0.4,angle=-90]{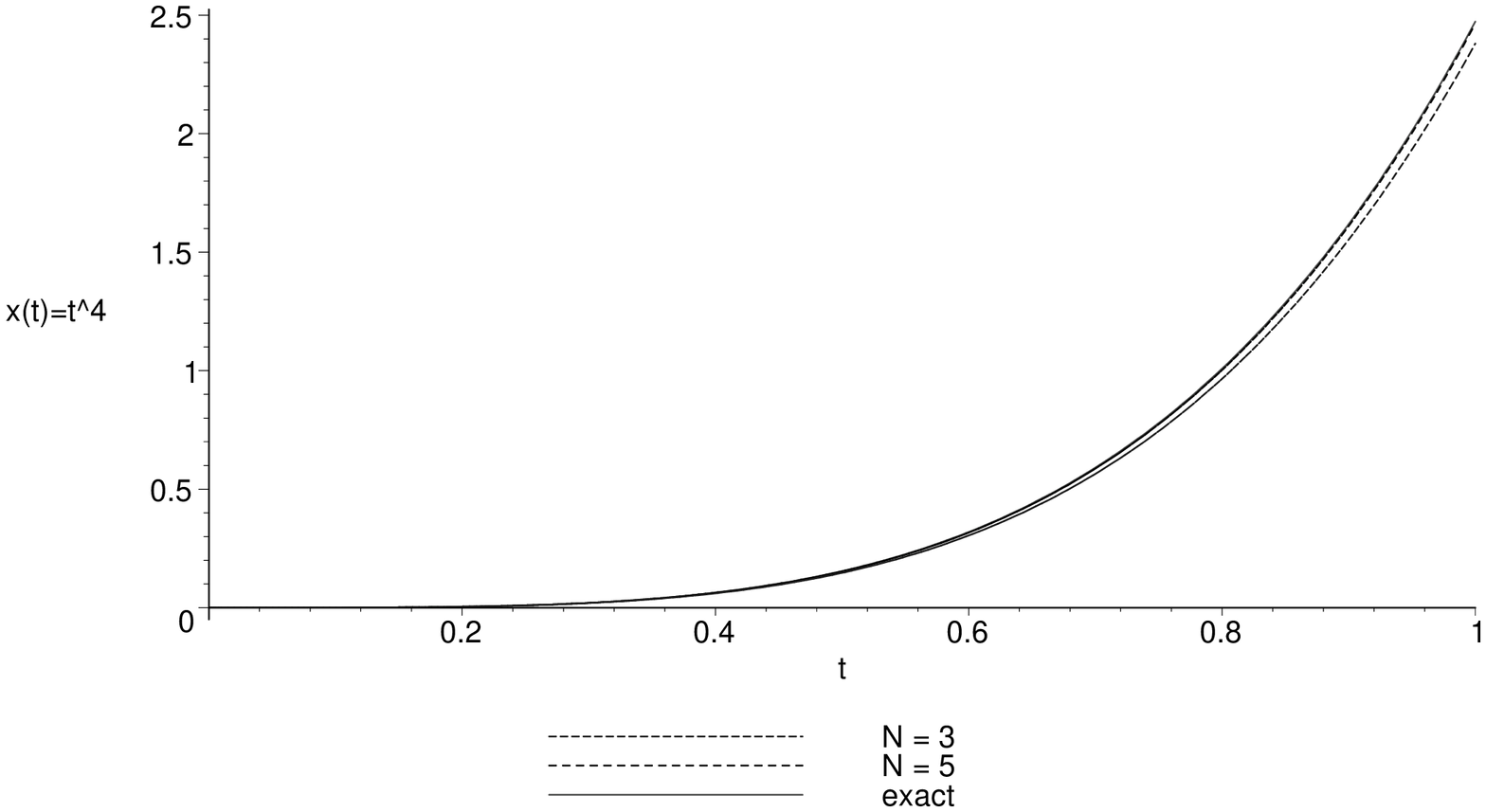}\\
\caption{Exact \eqref{ex1:LFD} and numerical approximations of the left Riemann--Liouville derivative
${_0D_t^{\a}} x(t)$ with $x(t)=t^4$ and $\a=(t+1)/4$ obtained from Theorem~\ref{Mainthm}
with $n=2$ and $N \in \{3,5\}$. The error \eqref{eq:err} is $E\approx 0.03294$ for $N=3$
and $E\approx 0.003976$ for $N=5$.}
\label{fig:2}
\end{center}
\end{figure}
\begin{figure}[!ht]
\begin{center}
\psfrag{x\(t\)=t^4}{\hspace*{-0.5cm}\tiny ${_0\mathbb{D}_t^{\a}} x(t)$}
\includegraphics[scale=0.4,angle=-90]{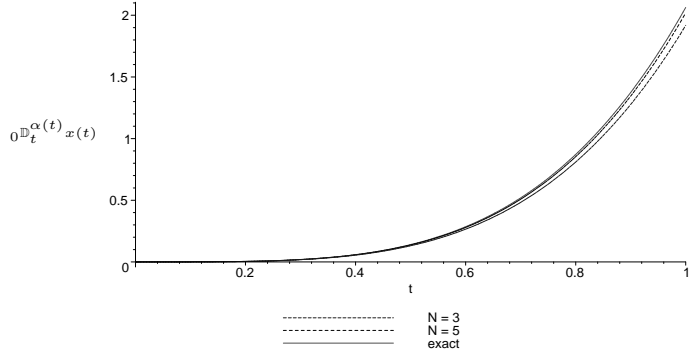}\\
\caption{Exact \eqref{ex1:LFD:M} and numerical approximations of the left Marchaud derivative
${_0\mathbb{D}_t^{\a}} x(t)$ with $x(t)=t^4$ and $\a=(t+1)/4$ obtained
from Theorem~\ref{MainthmM} with $n=2$ and $N \in \{3,5\}$.
The error \eqref{eq:err} is $E \approx 0.04919$ for $N=3$
and $E \approx 0.01477$ for $N=5$.}
\label{fig:3}
\end{center}
\end{figure}
\end{ex}


\subsection{Fractional differential equations of variable order}
\label{sec:FDEVO}

Consider the following fractional differential equation of variable order:
\begin{equation}
\label{eq:FDE:VO}
\begin{cases}
{_0\mathbb{D}_t^{\a}} x(t)+x(t)=\frac{1}{\Gamma(\frac{7-t}{4})}t^{\frac{3-t}{4}}+t,\\
x(0)=0,
\end{cases}
\end{equation}
with $\a=(t+1)/4$. It is easy to check that $\overline{x}(t)=t$ is a solution to \eqref{eq:FDE:VO}.
We exemplify how our Theorem~\ref{MainthmM}
may be applied in order to approximate the solution
of such type of problems. The main idea
is to replace all the fractional operators that appear in the differential
equation by a finite sum up to order $N$, involving integer derivatives only,
and, by doing so, to obtain a new system of standard ordinary differential equations
that is an approximation of the initial fractional variable order problem.
As the size of $N$ increases, the solution of the new system converges to the solution
of the initial fractional system. The procedure for \eqref{eq:FDE:VO} is the following.
First, we replace ${_0\mathbb{D}_t^{\a}} x(t)$ by
$$
{_0\mathbb{D}_t^{\a}} x(t) \approx  A(\a,N)t^{-\a}x(t)
+B(\a,N)t^{1-\a}x^{(1)}(t)+\sum_{k=2}^N C(\a,k)t^{1-k-\a}V_k(t),
$$
where
\begin{equation*}
\begin{split}
A(\a,N)&=\frac{1}{\Gamma(1-\a)}\left[1+\sum_{p=2}^N\frac{\Gamma(p-1+\a)}{\Gamma(\a)(p-1)!}\right],\\
B(\a,N)&=\frac{1}{\Gamma(2-\a)}\left[1+\sum_{p=1}^N\frac{\Gamma(p-1+\a)}{\Gamma(\a-1)p!}\right],\\
C(\a,k)&=\frac{\Gamma(k-1+\a)}{\Gamma(-\a)\Gamma(1+\a)(k-1)!},
\end{split}
\end{equation*}
and ${V}_k(t)$ is the solution of the system
\begin{equation*}
\begin{cases}
V^{(1)}_k(t)=(k-1)t^{k-2}x(t)\\
V_k(0)=0, \qquad k=2,3,\ldots,N.
\end{cases}
\end{equation*}
Thus, we get the approximated system of ordinary differential equations
\begin{equation}
\label{eq:stODE:syst}
\begin{cases}
\left[A(\a,N)t^{-\a}+1\right]x(t)+B(\a,N)t^{1-\a}x^{(1)}(t)
+\sum_{k=2}^N C(\a,k)t^{1-k-\a}V_k(t)\\
\quad =\frac{1}{\Gamma(\frac{7-t}{4})}t^{\frac{3-t}{4}}+t,\\
V^{(1)}_k(t)=(k-1)t^{k-2}x(t),  \qquad k=2,3,\ldots,N,\\
x(0)=0, \\
V_k(0)=0, \qquad k=2,3,\ldots,N.
\end{cases}
\end{equation}
Now we apply any standard technique to solve the system
of ordinary differential equations \eqref{eq:stODE:syst}.
We used the command \textsf{dsolve} of \textsf{Maple}.
In Figure~\ref{fig:4} we find the graph of the approximation
$\tilde{x}_3(t)$ to the solution of problem \eqref{eq:FDE:VO},
obtained solving \eqref{eq:stODE:syst} with $N=3$.
The Table~\ref{tab1} gives some numerical values of such approximation,
illustrating numerically the fact that the approximation $\tilde{x}_3(t)$
is already very close to the exact solution $\overline{x}(t)=t$ of \eqref{eq:FDE:VO}.
In fact the plot of $\tilde{x}_3(t)$ in Figure~\ref{fig:4} is visually
indistinguishable from the plot of $\overline{x}(t)=t$.
\begin{figure}[!ht]
\begin{center}
\psfrag{x\(t\)}{\tiny $\tilde{x}_3(t)$}
\includegraphics[scale=0.4,angle=-90]{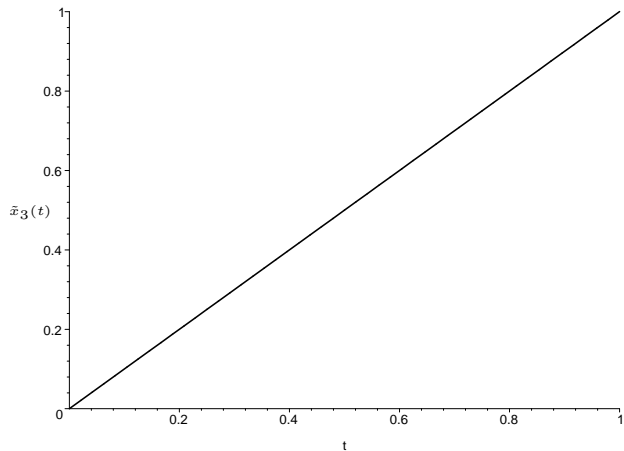}\\
\caption{Approximation $\tilde{x}_3(t)$ to the exact solution
$\overline{x}(t)=t$ of the fractional differential equation \eqref{eq:FDE:VO},
obtained from the application of Theorem~\ref{MainthmM}, that is,
obtained solving \eqref{eq:stODE:syst} with $N = 3$.}
\label{fig:4}
\end{center}
\end{figure}
\begin{table}[!ht]
$$
\begin{array}{|c|c|c|c|c|c|c|}\hline
t & 0.2 & 0.4 & 0.6 & 0.8 & 1\\
\hline
\tilde{x}_3(t) & 0.20000002056 & 0.40000004031 & 0.60000009441 & 0.80000002622 & 1.0000001591\\
\hline
\end{array}
$$
\caption{Some numerical values of the solution $\tilde{x}_3(t)$
of \eqref{eq:stODE:syst} with $N = 3$, very close to the
values of the solution $\overline{x}(t)=t$ of the fractional
differential equation of variable order \eqref{eq:FDE:VO}.}
\label{tab1}
\end{table}


\subsection{Fractional variational calculus of variable order}

We now exemplify how the expansions obtained in Section~\ref{sec:theorems}
are useful to approximate solutions of fractional problems
of the calculus of variations \cite{Malinowska}.
The fractional variational calculus of variable order
is a recent subject under strong current development \cite{Od0,Od1,Od2,MyID:268}.
So far, only analytical methods to solve fractional problems
of the calculus of variations of variable order
have been developed in the literature,
which consist in the solution of fractional Euler--Lagrange
differential equations of variable order \cite{Od0,Od1,Od2,MyID:268}.
In most cases, however, to solve analytically such fractional differential
equations is extremely hard or even impossible, so numerical/approximating
methods are needed. Our results provide two approaches to this issue.
The first was already illustrated in Section~\ref{sec:FDEVO} and
consists in approximating the necessary optimality conditions proved in
\cite{Od0,Od1,Od2,MyID:268}, which are nothing else than
fractional differential equations of variable order.
The second approach is now considered. Similarly to Section~\ref{sec:FDEVO},
the main idea here is to replace the fractional operators of variable order
that appear in the formulation of the variational problem
by the corresponding expansion of Section~\ref{sec:theorems},
which involves only integer order derivatives. By doing it, we reduce
the original problem to a classical optimal control problem,
whose extremals are found by applying the celebrated
Pontryagin maximum principle \cite{Pontryagin}.
We illustrate this method with a concrete example. Consider the functional
\begin{equation}
\label{eq:ex:VP}
J(x)=\int_0^1\left[{_0\mathbb{D}_t^{\a}} x(t)
-\frac{1}{\Gamma(\frac{7-t}{4})}t^{\frac{3-t}{4}}\right]^2dt,
\end{equation}
with fractional order $\a=(t+1)/4$, subject to the boundary conditions
\begin{equation}
\label{eq:ex:BC}
x(0)=0, \quad x(1)=1.
\end{equation}
Since $J(x)\geq 0$ for any admissible function $x$ and taking $\overline{x}(t)=t$,
which satisfies the given boundary conditions \eqref{eq:ex:BC},
one has $J(\overline{x})=0$, we conclude that $\overline{x}$
gives the global minimum to the fractional problem of the
calculus of variations that consists in minimizing functional \eqref{eq:ex:VP}
subject to the boundary conditions \eqref{eq:ex:BC}.
The numerical procedure is now explained. Since we have two boundary conditions,
we replace ${_0\mathbb{D}_t^{\a}} x(t)$ by the expansion given in Theorem~\ref{MainthmM}
with $n=1$ and a variable size $N\geq 2$. The approximation becomes
\begin{equation}
\label{eq:appr:PCV}
{_0\mathbb{D}_t^{\a}} x(t) \approx  A(\a,N)t^{-\a}x(t)
+B(\a,N)t^{1-\a}x^{(1)}(t)+\sum_{k=2}^N C(\a,k)t^{1-k-\a}V_k(t).
\end{equation}
Using \eqref{eq:appr:PCV}, we approximate the initial problem
by the following one: to minimize
\begin{multline*}
\tilde J(x)=\int_0^1\Biggl[ A(\a,N)t^{-\a}x(t)
+B(\a,N)t^{1-\a}x^{(1)}(t)\\
+\sum_{k=2}^N C(\a,k)t^{1-k-\a}V_k(t)
- \frac{1}{\Gamma(\frac{7-t}{4})}t^{\frac{3-t}{4}}\Biggr]^2dt
\end{multline*}
subject to
$$
V^{(1)}_k(t)=(k-1)t^{k-2}x(t),
\quad V_k(0)=0, \qquad k=2,\ldots,N,
$$
and
$$
x(0)=0, \quad x(1)=1,
$$
where $\a=(t+1)/4$. This dynamic optimization problem has a system of ordinary
differential equations as a constraint, so it is natural
to solve it as an optimal control problem.
For that, define the control $u$ by
$$
u(t) = A(\a,N)t^{-\a}x(t)+B(\a,N)t^{1-\a}x^{(1)}(t)
+\sum_{k=2}^N C(\a,k)t^{1-k-\a}V_k(t).
$$
We then obtain the control system
$$
x^{(1)}(t)=B^{-1}t^{\a-1}u(t)-AB^{-1}t^{-1}x(t)
-\sum_{k=2}^N B^{-1}C_kt^{-k}V_k(t)
:=f\left(t,x(t),u(t),V(t)\right),
$$
where, for simplification,
$$
A=A(\a,N), \quad B= B(\a,N),\quad C_k=C(\a,k),
\quad \mbox{and} \quad V(t)=(V_2(t),\ldots,V_N(t)).
$$
In conclusion, we wish to minimize the functional
$$
\tilde{J}(x,u,V)=\int_0^1\left[ u(t)
- \frac{1}{\Gamma(\frac{7-t}{4})}t^{\frac{3-t}{4}}\right]^2dt
$$
subject to the first-order dynamic constraints
$$
\begin{cases}
x^{(1)}(t)=f(t,x,u,V),\\
V^{(1)}_k(t)=(k-1)t^{k-2}x(t), \quad k=2,\ldots,N,
\end{cases}
$$
and the boundary conditions
$$
\begin{cases}
x(0)=0,\\
x(1)=1,\\
V_k(0)=0, \qquad k=2,\ldots,N.
\end{cases}
$$
In this case, the Hamiltonian is given by
$$
H(t,x,u,V,\lambda)=\left[ u - \frac{1}{\Gamma(\frac{7-t}{4})}t^{\frac{3-t}{4}}\right]^2
+\lambda_1 f(t,x,u,V)+\sum_{k=2}^N\lambda_k (k-1)t^{k-2}x
$$
with the adjoint vector $\lambda=(\lambda_1,\lambda_2,\ldots,\lambda_N)$ \cite{Pontryagin}.
Following the classical optimal control approach of Pontryagin \cite{Pontryagin},
we have the following necessary optimality conditions:
$$
\frac{\partial H}{\partial u}=0, \quad x^{(1)}
=\frac{\partial H}{\partial \lambda_1},
\quad V_p^{(1)}=\frac{\partial H}{\partial \lambda_p},  \quad \lambda_1^{(1)}
=-\frac{\partial H}{\partial x}, \quad \lambda_p^{(1)}
=-\frac{\partial H}{\partial V_p},
$$
that is, we need to solve the system of differential equations
\begin{equation}
\label{eq:CNP:SH}
\begin{cases}
x^{(1)}(t)= \frac{B^{-1}}{\Gamma(\frac{7-t}{4})}
-\frac{1}{2}B^{-2}t^{2\a-2}\lambda_1(t)
-AB^{-1}t^{-1}x(t)-\sum_{k=2}^NB^{-1}C_kt^{-k}V_k(t),\\
V_k^{(1)}(t)=(k-1)t^{k-2}x(t),\quad k=2,\ldots,N,\\
\lambda_1^{(1)}(t)=AB^{-1}t^{-1}\lambda_1-\sum_{k=2}^N(k-1)t^{k-2}\lambda_k(t),\\
\lambda_k^{(1)}(t)=B^{-1}C_kt^{-k}\lambda_1,\quad k=2,\ldots,N,
\end{cases}
\end{equation}
subject to the boundary conditions
\begin{equation}
\label{eq:CNP:BC}
\begin{cases}
x(0)=0,\\
V_k(0)=0,\quad k=2,\ldots,N,\\
x(1)=1,\\
\lambda_k(1)=0,\quad k=2,\ldots,N.
\end{cases}
\end{equation}
Figure~\ref{fig:5} plots the numerical approximation $\tilde{x}_2(t)$
to the global minimizer $\overline{x}(t)=t$ of the variable order
fractional problem of the calculus of variations
\eqref{eq:ex:VP}--\eqref{eq:ex:BC}, obtained solving
\eqref{eq:CNP:SH}--\eqref{eq:CNP:BC} with $N=2$.
The approximation $\tilde{x}_2(t)$ is already visually indistinguishable
from the exact solution $\overline{x}(t)=t$, and we do not
increase the value of $N$. The effectiveness of our approach is also illustrated
in Table~\ref{tab2}, where some numerical values of the approximation
$\tilde{x}_2(t)$ are given.
\begin{figure}[!ht]
\begin{center}
\psfrag{x\(t\)}{\tiny $\tilde{x}_{2}(t)$}
\includegraphics[scale=0.4,angle=-90]{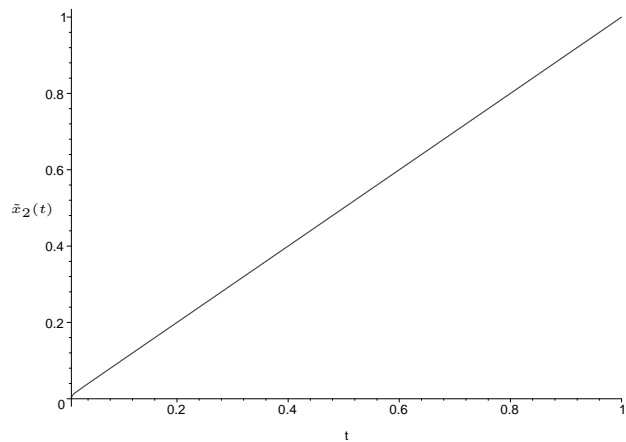}\\
\caption{Approximation $\tilde{x}_2(t)$ to the exact solution
$\overline{x}(t)=t$ of the fractional problem
of the calculus of variations \eqref{eq:ex:VP}--\eqref{eq:ex:BC},
obtained from the application of Theorem~\ref{MainthmM} and the
classical Pontryagin maximum principle, that is,
obtained solving \eqref{eq:CNP:SH}--\eqref{eq:CNP:BC} with $N = 2$.}
\label{fig:5}
\end{center}
\end{figure}
\begin{table}[!ht]
$$
\begin{array}{|c|c|c|c|c|c|c|}\hline
t & 0.2&0.4&0.6&0.8&1\\
\hline
\tilde{x}_2(t) & 0.1998346692 & 0.3999020706 & 0.5999392936 & 0.7999708526 & 1.0000000000 \\
\hline
\end{array}
$$
\caption{Some numerical values of the solution $\tilde{x}_2(t)$
of \eqref{eq:CNP:SH}--\eqref{eq:CNP:BC} with $N = 2$, close to the
values of the global minimizer $\overline{x}(t)=t$ of the fractional
variational problem of variable order \eqref{eq:ex:VP}--\eqref{eq:ex:BC}.}
\label{tab2}
\end{table}


\section*{Acknowledgments}

Work supported by FEDER funds through COMPETE -- Operational
Programme Factors of Competitiveness (``Programa Operacional Factores de Competitividade'')
and by Portuguese funds through the {\it Center for Research
and Development in Mathematics and Applications}
(University of Aveiro) and the Portuguese Foundation for Science and Technology
(``FCT -- Funda\c{c}\~{a}o para a Ci\^{e}ncia e a Tecnologia''), within project
PEst-C/MAT/UI4106/2011 with COMPETE number FCOMP-01-0124-FEDER-022690.
Torres was also supported by EU funding under the
7th Framework Programme FP7-PEOPLE-2010-ITN,
grant agreement number 264735-SADCO.




\begin{thebibliography}{99}

\bibitem{AlmeidaSamko}
A. Almeida\ and\ S. Samko,
Fractional and hypersingular operators in variable
exponent spaces on metric measure spaces,
Mediterr. J. Math. {\bf 6} (2009), no.~2, 215--232.

\bibitem{Atanackovic1}
T. M. Atanackovic, M. Janev, S. Pilipovic\ and\ D. Zorica,
An expansion formula for fractional derivatives of variable order,
Cent. Eur. J. Phys. (2013), DOI: 10.2478/s11534-013-0243-z

\bibitem{Coimbra}
C. F. M. Coimbra,
Mechanics with variable-order differential operators,
Ann. Phys. (8) {\bf 12} (2003), no.~11-12, 692--703.

\bibitem{Coimbra2}
C. F. M. Coimbra, C. M. Soon\ and\ M. H. Kobayashi,
The variable viscoelasticity operator,
Annalen der Physik {\bf 14} (2005), 378--389.

\bibitem{Dalir}
M. Dalir\ and\ M. Bashour,
Applications of fractional calculus,
Appl. Math. Sci. (Ruse) {\bf 4} (2010), no.~21-24, 1021--1032.

\bibitem{Diaz}
G. Diaz\ and\ C. F. M. Coimbra,
Nonlinear dynamics and control of a variable order oscillator
with application to the van der Pol equation,
Nonlinear Dynam. {\bf 56} (2009), no.~1-2, 145--157.

\bibitem{Lorenzo}
C. F. Lorenzo\ and\ T. T. Hartley,
Variable order and distributed order fractional operators,
Nonlinear Dynam. {\bf 29} (2002), no.~1-4, 57--98.

\bibitem{Ma}
S. Ma, Y. Xu\ and\ W. Yue,
Numerical solutions of a variable-order fractional financial system,
J. Appl. Math. {\bf 2012} (2012), Art. ID 417942, 14~pp.

\bibitem{Machado}
J. A. T. Machado, M. F. Silva, R. S. Barbosa, I. S. Jesus,
C. M. Reis, M. G. Marcos\ and\ A. F. Galhano,
Some applications of fractional calculus in engineering,
Math. Probl. Eng. {\bf 2010} (2010), Art. ID 639801, 34~pp.

\bibitem{Malinowska}
A. B. Malinowska\ and\ D. F. M. Torres,
{\it Introduction to the fractional calculus of variations},
Imp. Coll. Press, London, 2012.

\bibitem{Od0}
T. Odzijewicz, A. B. Malinowska\ and\ D. F. M. Torres,
Variable order fractional variational calculus for double integrals,
Proceedings of the 51st IEEE Conference on Decision and Control,
December 10--13, 2012, Maui, Hawaii, Art. no. 6426489 (2012), pp.~6873--6878.
{\tt arXiv:1209.1345}

\bibitem{Od1}
T. Odzijewicz, A. B. Malinowska\ and\ D. F. M. Torres,
Fractional variational calculus of variable order,
Advances in Harmonic Analysis and Operator Theory,
The Stefan Samko Anniversary Volume (Eds: A. Almeida, L. Castro, F.-O. Speck),
Operator Theory: Advances and Applications, Vol.~229, 291--301, Springer, 2013.
{\tt arXiv:1110.4141}

\bibitem{Od2}
T. Odzijewicz, A. B. Malinowska\ and\ D. F. M. Torres,
Noether's theorem for fractional variational problems of variable order,
Cent. Eur. J. Phys. (2013), DOI: 10.2478/s11534-013-0208-2
{\tt arXiv:1303.4075}

\bibitem{MyID:268}
T. Odzijewicz, A. B. Malinowska\ and\ D. F. M. Torres,
A generalized fractional calculus of variations,
Control Cybernet. {\bf 42} (2013), no.~2, in press.
{\tt arXiv:1304.5282}

\bibitem{Pontryagin}
L. S. Pontryagin, V. G. Boltyanskii, R. V. Gamkrelidze\ and\ E. F. Mishchenko,
{\it The mathematical theory of optimal processes},
Translated from the Russian by K. N. Trirogoff;
edited by L. W. Neustadt Interscience Publishers
John Wiley \& Sons, Inc.\, New York, 1962.

\bibitem{Pooseh0}
S. Pooseh, R. Almeida\ and\ D. F. M. Torres,
Approximation of fractional integrals by means of derivatives,
Comput. Math. Appl. {\bf 64} (2012), no.~10, 3090--3100.
{\tt arXiv:1201.5224}

\bibitem{Pooseh1}
S. Pooseh, R. Almeida\ and\ D. F. M. Torres,
Numerical approximations of fractional derivatives with applications,
Asian J. of Control  {\bf 15} (2013), no.~3, 698--712.
{\tt arXiv:1208.2588}

\bibitem{Ramirez1}
L. E. S. Ramirez\ and\ C. F. M. Coimbra,
On the selection and meaning of variable order operators for dynamic modeling,
Int. J. Differ. Equ. {\bf 2010} (2010), Art. ID 846107, 16~pp.

\bibitem{Ramirez2}
L. E. S. Ramirez\ and\ C. F. M. Coimbra,
On the variable order dynamics of the nonlinear wake caused by a sedimenting particle,
Phys. D {\bf 240} (2011), no.~13, 1111--1118.

\bibitem{Samko}
S. G. Samko,
Fractional integration and differentiation of variable order,
Anal. Math. {\bf 21} (1995), no.~3, 213--236.

\bibitem{samko0}
S. G. Samko, A. A. Kilbas\ and\ O. I. Marichev,
{\it Fractional integrals and derivatives},
translated from the 1987 Russian original,
Gordon and Breach, Yverdon, 1993.

\bibitem{SamkoRoss}
S. G. Samko\ and\ B. Ross,
Integration and differentiation to a variable fractional order,
Integral Transform. Spec. Funct. {\bf 1} (1993), no.~4, 277--300.

\bibitem{Valerio}
D. Val\'{e}rio\ and\ J. S. Costa,
Variable-order fractional derivatives and their numerical approximations,
Signal Process. {\bf 91} (2011), 470-–483.

\bibitem{Zhuang}
P. Zhuang, F. Liu, V. Anh\ and\ I. Turner,
Numerical methods for the variable-order fractional
advection-diffusion equation with a nonlinear source term,
SIAM J. Numer. Anal. {\bf 47} (2009), no.~3, 1760--1781.

\end{thebibliography}
\end{document}